\documentclass[11pt,a4paper]{amsart}
\usepackage{amssymb}
\usepackage{mathpazo}
\usepackage[T1]{fontenc}
\usepackage{mathrsfs}
\usepackage[all]{xy}
\usepackage[headings]{fullpage}
\usepackage[colorlinks=true,linkcolor=blue, citecolor=blue, urlcolor=blue,%
pagebackref=true]{hyperref}
\makeatletter
\@addtoreset{equation}{section}
\def\theequation{\thesection.\@arabic \c@equation}
\def\theenumi{\@alph\c@enumi}
\makeatother
\theoremstyle{plain}
\newtheorem{theorem}[equation]{Theorem}
\newtheorem{lemma}[equation]{Lemma}
\newtheorem{corollary}[equation]{Corollary}
\newtheorem{proposition}[equation]{Proposition}
\theoremstyle{definition}
\newtheorem{definition}[equation]{Definition}
\newtheorem{discussion}[equation]{Discussion}
\newenvironment{discussionbox}[1][]{%
    \begin{discussion}[#1]\pushQED{\qed}}{\popQED \end{discussion}}
\newcommand{\fraka}{{\mathfrak a}}
\newcommand{\frakb}{{\mathfrak b}}
\newcommand{\calE}{\mathcal E}
\newcommand{\calF}{\mathcal F}
\newcommand{\calI}{\mathcal I}
\newcommand{\calK}{\mathcal K}
 \let\strSh\calO
\newcommand{\frakp}{{\mathfrak p}}
\newcommand{\naturals}{\mathbb{N}}
\newcommand{\ints}{\mathbb{Z}}
\newcommand{\complex}{\mathbb{C}}
\DeclareMathOperator{\rank}{rk}
\DeclareMathOperator{\coker}{coker}
\DeclareMathOperator{\projective}{\mathbb{P}}
\DeclareMathOperator{\projdim}{pd}
\DeclareMathOperator{\height}{ht}
\DeclareMathOperator{\codim}{codim}
\DeclareMathOperator{\Hom}{Hom}
\DeclareMathOperator{\Tor}{Tor}
\DeclareMathOperator{\Supp}{Supp}
\DeclareMathOperator{\Spec}{Spec}
\DeclareMathOperator{\Proj}{Proj}
\DeclareMathOperator{\Min}{Min}
\DeclareMathOperator{\depth}{depth}
\DeclareMathOperator{\homology}{H}
\newcommand{\define}[1]{\emph{#1}}
\newcommand{\minus}{\ensuremath{\smallsetminus}}
\DeclareMathOperator{\image}{Im}
\DeclareMathOperator{\Ass}{Ass}
\DeclareMathOperator{\Ann}{Ann}
\DeclareMathOperator{\socle}{soc}
\DeclareMathOperator{\Pic}{Pic}
\DeclareMathOperator{\Ext}{Ext}
\DeclareMathOperator{\Frac}{Frac}
\newcommand{\sheafHom}{\mathcal{H}om}
\newcommand{\sheafExt}{\mathcal{E}xt}
\begin{document}
\title{On codimension-two subcanonical varieties inside $\projective^n$}
\author{Manoj Kummini}
\address{Chennai Mathematical Institute, Siruseri, Tamilnadu 603103. India}
\email{mkummini@cmi.ac.in}

\author{Abhiram Subramanian}
\address{Chennai Mathematical Institute, Siruseri, Tamilnadu 603103. India}
\email{abhiram@cmi.ac.in}

\thanks{Both authors were partially supported by an Infosys Foundation
grant. The second author was supported by a PhD fellowship from the
National Board for Higher Mathematics, India.}

\subjclass[2020]{Primary: 13C40; Secondary: 14M10} 

\begin{abstract}
Let $X \subseteq \projective^n, n \geq 4$ be a codimension-two subcanonical
local complete intersection variety with ideal sheaf $\calI_X$. Let $a_X
\in \ints$ be such that
$\omega_X = \strSh_X(a_X)$. Assume that there exists 
$\displaystyle j \leq \frac{a_X+n+2}{2}$ such that 
$\Gamma(\calI_X(j)) \neq 0$.
We prove some sufficient conditions on the first deficiency module
$\homology^1_*(\calI_X)$ that ensures that $X$ is a complete intersection.
We also show that smooth codimension-two $3$-Buchsbaum varieties inside 
$\projective^n, n \geq 6$ are complete intersections.
\end{abstract}

\maketitle

\section{Introduction}

Hartshorne's conjecture~\cite[p.~1017]{HartshorneVarSmallCodim1974} asserts
that over an algebraically closed field $\Bbbk$,
a smooth subvariety $X \subseteq  \projective_\Bbbk^n$
with $\dim X > \frac {2n}3$ is a complete intersection.
He remarks~\cite[p.~1022]{HartshorneVarSmallCodim1974} that,
when the codimension of $X$ is two,
even though the conjecture is stated only for $n \geq 7$ (and,
consequently, $\dim X \geq 5$), there are no known examples of
four-dimensional smooth subvarieties of $\projective_\Bbbk^6$ that are not
complete intersections. 
If $X$ is a codimension-two smooth subvariety
of $\projective_\Bbbk^n$, with $n \geq 6$, then $\Pic X \simeq \ints$,
generated
by $\strSh_X(1)$~\cite[Theorem~2.2(d)]{HartshorneVarSmallCodim1974}.
Hence such varieties are subcanonical, i.e., 
there exists $a_X \in \ints$ such that $\omega_X = \strSh_X(a_X)$.
Building on the subsequent work on
this conjecture, Ellia~\cite[Conjecture~2]{ElliaSurvey2000} proposed a
stronger
conjecture in characteristic $0$: if $X \subseteq \projective^n_\complex$,
$n \geq 5$, is a smooth subcanonical\footnote
{When $n=5$, one must require that $X$ is subcanonical, since
the ideal of maximal minors of a $2 \times 3$ matrix of indeterminates
defines a smooth codimension-two subvariety of $\projective_\Bbbk^5$.}
variety of codimension two, then $X$ is a complete 
intersection\footnote{Ballico and
Chiantini~\cite[p.~100]{BallicoChiantiniSmSubcan1983} also mention that 
every known example of
codimension-two smooth subcanonical varieties in
$\projective^n_\complex$, $n > 4$, is a complete intersection.}.

In this paper, we prove some results about 
codimension-two subcanonical local complete intersections.
Let $\Bbbk$ be a field and $n \geq 4$ an integer.
Let $X \subseteq \projective^n := \projective^n_\Bbbk$ be a local complete
intersection subscheme with $\codim_{\projective^n}X = 2$.
Let $R = \Bbbk[x_0, \ldots, x_n]$ be a polynomial ring in the
indeterminates $x_0, \ldots, x_n$, graded with $\deg x_i = 1$ for all $i$.
Write $R_+$ for the homogeneous maximal ideal of $R$.
Write $\calI_X$ for the ideal sheaf of $X$ inside $\projective^n$.
Let $I_X = \Gamma_*(\calI_X)$. Let $d_1 \leq d_2 \leq \cdots \leq d_m$ be
the degrees of a minimal homogeneous generating set of $I_X$.
When $X$ is a Cohen-Macaulay scheme,
we say that $X$ is \define{subcanonical} if there exists $a_X \in \ints$
such that $\omega_X = \strSh_X(a_X)$;
here $\omega_X :=
\sheafExt_{\strSh_{\projective^n}}^2(\strSh_X,
\strSh_{\projective^n}(-n-1))$.
The \define{first deficiency module} (also called the \emph{Rao module})
$M_1(X)$ of $X$ is 
$\homology^1_*(\calI_X)$, which equals $\homology_{R_+}^1(R/I_X)$;
see~\cite[\S 1.2]{MiglioreLiaisonBook1998}.
It is a graded $R$-module.
For $R$-modules $M$, $\socle(M)$ denotes the \define{socle} of $M$, i.e.,
$(0:_M R_+)$.

We now list our main results.

\begin{theorem}
\label{theorem:minHOne}
Suppose that $X \subseteq \projective^n$, $n \geq 4$,
is a codimension-two subcanonical
local complete intersection scheme, with $\omega_X = \strSh_X(a_X)$.
Assume that $2d_1 \leq a_X+n+2$ and that
\[
\min\{j \mid \homology^1(\projective^n, \calI_X(j ) ) \neq 0 \} \geq d_1-1.
\]
Then $X$ is a complete intersection.
\end{theorem}

Theorem~\ref{theorem:minHOne} and Corollary~\ref{corollary:gap} below  was
proved for $n=3$ by Chiantini and
Valabrega~\cite[Theorem~6 and
Corollary~7]{ChiantiniVallabregaSubcanCurves1987}. 
The main objective of the first part of the paper is to extend it to higher
dimensions.

\begin{theorem}
\label{theorem:CIorNegSocZero}
Suppose that $X \subseteq \projective^n$, $n \geq 4$,
is a codimension-two subcanonical
local complete intersection scheme, with $\omega_X = \strSh_X(a_X)$.
Assume that $2d_1 \leq a_X+n+2$ and that $X$ is not a complete intersection.
Then 
\[
(\socle(M_1(X)))_j = 0
\;\text{for all}\;
j < \max \{ a_X+n+2-d_1, d_3-2\}.
\]
\end{theorem}

Let $k$ be a positive integer.
A codimension-two projective subscheme $X \subseteq \projective^n$ is
\define{$k$-Buchsbaum} if 
$(R_+)^k\homology^i_*(\projective^n, \calI_{X \cap L}) = 0$
for every general linear subspace $L \subseteq \projective^n$
and for all $1 \leq i \leq \dim L-2$.
Ellia and Sarti~\cite [Proposition~13]{ElliaSartikBuchs1999} showed that if
$X$ is a smooth codimension-two $3$-Buchsbaum
subscheme of $\projective^n$, $n \geq 6$, with $\omega_X = \strSh_X(a_X )$
and $2d_1 \geq a_X + n + 2$, then $X$ is a complete intersection.
We use our above results to complement this, and obtain the following:

\begin{corollary}
\label{corollary:3Buchs}
Let $X$ be a smooth codimension-two $3$-Buchsbaum
subscheme of $\projective^n$, $n \geq 6$. 
Then $X$ is a complete intersection.
\end{corollary}

Chiantini and 
Valabrega~\cite[Theorems~1.3, 1.6]{ChiantiniVallabregaSubcanCurves1983}
showed that if $X$ is a smooth subcanonical curve in
$\projective^3_\complex$ with 
$\homology^1(\projective^3_\complex, \calI_X(j)) = 0$ for some
special $j$ (determined from $a_X$), then $X$ is a complete intersection.
In~\cite[p.~388]{ChiantiniVBReflShLowCodimVars2006}
Chiantini asks if a similar result holds for codimension-two subcanonical
local complete intersection schemes 
$X \subseteq \projective^n$, $n \geq 4$.
(The question is stated in an equivalent form for rank-two vector bundles.)
The following corollary gives an affirmative case of this question,
under the assumption that $2d_1 \leq a_X+n+2$.

\begin{corollary}
\label{corollary:gap}
Assume that $2d_1 \leq a_X+n+2$. If $\homology^1(\projective^n, \calI_X(j))
= 0$ for some $j$ with 
\[
d_1-2 \leq j \leq  \max \{ a_X+n+2-d_1, d_3-2\},
\]
then $X$ is a complete intersection.
\end{corollary}

Given a codimension-two subcanonical local complete intersection $X$ as
above, we have an exact sequence
\begin{equation}
\label{equation:sesCalE} 
0 \to \strSh_{\projective^n}(-(a_X+n+1)) \to \calE \to \calI_X \to 0
\end{equation}
where $\calE$ is a rank-two vector bundle on $\projective^n$.
(See Proposition~\ref{proposition:lcisubcan}.) 
$\calE$ splits if and only if $X$ is a complete intersection.

The hypothesis that $2d_1 \leq a_X + n+2$ is related to the stability of
$\calE$.
Note that $d_1 \leq 2d_1-1 \leq a_X + n + 1$.
Therefore, applying $\homology^0(-)$ to~\eqref{equation:sesCalE} we see
that 
\[
d_1 = \min \{j \mid \homology^0(\calE(j)) \neq 0 \}.
\]
Therefore $\calE$ contains the line bundle $\strSh_{\projective^n}(-d_1)$.
Since $c_1(\calE) = -(a_X+n+1)$~\cite[Chapter~1, 5.1.1]{OSSvb2011},
we see that $\calE$ is stable if $2d_1 = a_X+n+2$.
Therefore we get the following consequences of
Theorem~\ref{theorem:minHOne} and Corollary~\ref{corollary:gap}:

\begin{corollary}
Suppose that $X \subseteq \projective^n$, $n \geq 4$,
is a codimension-two subcanonical
local complete intersection scheme, with $\omega_X = \strSh_X(a_X)$
and $2d_1 = a_X+n+2$. Then 
$\homology^1(\projective^n, \calI_X(j ) ) \neq 0$ 
for all $j$ with $d_1-2 \leq j \leq  \max \{ a_X+n+2-d_1, d_3-2\}$.
\end{corollary}

For more results similar to the above ones, see 
\cite{ElliaFrancoCodimTwo2002},
\cite{ElliaSurvey2000},
\cite{ElliaSartikBuchs1999},
\cite{ChiantiniVallabregaSubcanCurves1983},
\cite{ChiantiniVallabregaSubcanCurves1987},
\cite{ChiantiniVBReflShLowCodimVars2006} and
\cite{KumarRaoBuchsbaumBdles2000} and their bibliography.

This paper is organized as follows.
Section~\ref{section:prelim} discusses some preliminaries.
If the given local complete intersection scheme is not a complete
intersection, we construct a new 
local complete intersection scheme in Section~\ref{section:construction}.
Proofs of the above results are in Section~\ref{section:proofs}.
Section~\ref{section:compare} gives an example that compares our hypothesis
with the hypotheses of some earlier papers on this topic.
In Section~\ref{section:hmbundle}, we contrast the cohomology table
of the Horrocks-Mumford bundle on $\projective^4_\complex$ with our hypotheses.

\section{Preliminaries}
\label{section:prelim}

In this section we collect some known results.
We prove a few statements; they might be known, but we have included a
proof for completeness. 
We start with a result about codimension-two subcanonical subschemes of
$\projective^n$, $n \geq 3$;
see~\cite[Proposition~6.1]{HartshorneVarSmallCodim1974}, \cite[Chapter~I,
\S\S~5.1 and~5.2]{OSSvb2011}.
\begin{proposition}
\label{proposition:lcisubcan}
Let $Z \subseteq \projective^n$, $n \geq 3$ be a codimension-two closed
subscheme, with ideal sheaf $\calI_Z$.
Then the following are equivalent:

\begin{enumerate}

\item
$Z$ is local complete intersection subcanonical scheme with
$\omega_Z = \strSh_Z(a_Z)$, $a_Z \in \ints$.

\item
There exists a rank-two locally free sheaf $\calE$ on 
$\projective^n$ and an exact sequence
\[
0 \to \strSh_{\projective^n}(-(a_Z+n+1)) \to \calE \to \calI_Z \to 0.
\]
\end{enumerate}
Further, $Z$ is a complete intersection if and only if $\calE$ splits,
i.e., it is isomorphic to a direct sum of two line bundles.
\end{proposition}

The cohomology table of a rank-two locally free sheaf $\calE$ 
(i.e., the collection 
$\rank_\Bbbk \homology^i(\calE(j))$ $0 \leq i \leq n$, $j \in \ints$)
that appears in the exact sequence in the above proposition satisfies a
duality, which we now explain.
Let $\calE$ be a rank-two locally free sheaf on $\projective^n$.
Then the dual 
\[
\calE^* := 
\sheafHom_{\strSh_{\projective^n}}(\calE, \strSh_{\projective^n})
\simeq
\calE \otimes_{\strSh_{\projective^n}} (\det \calE)^{-1}.
\]
Assume that $\calE$ is a sheaf that appears in the middle of the exact
sequence in Proposition~\ref{proposition:lcisubcan}.
Then $\det \calE = \strSh_{\projective^n}(-(a_Z + n + 1 ) )$.
Hence $\calE^* = \calE(a_Z + n + 1 )$.
Using Serre duality~\cite[Chapter~III Corollary~7.7]{HartAG} 
we get the following isomorphism:
\begin{equation}
\label{equation:duality}
\homology^i(\calE(j)) \simeq 
\homology^{n-i}(\calE(j)^*(-n-1))' 
\simeq \homology^{n-i}(\calE(a_Z - j))' 
\end{equation}
where $(-)' = \Hom_\Bbbk(-, \Bbbk)$.

We need a relation between the degrees of the minimal generators of a
saturated graded $R$-ideal $J$ and the socle of the first deficiency module
$M_1(Z)$ of the scheme $Z$ defined by $J$. We first prove a lemma.

\begin{lemma}
\label{lemma:tid3}
Let $J$ be a graded $R$-ideal generated minimally in degrees $\delta_1
\leq \cdots \leq \delta_\ell$, with $\ell \geq 3$.
Write $t_i = \min \{j \mid \Tor_i^R(\Bbbk, J)_j \neq 0 \}$.
Then for all $i \geq 2$, $t_i \geq \delta_3+i$.
\end{lemma}

\begin{proof}
If $F_\bullet$ is a minimal graded free resolution of $J$, then 
$t_i = \min\{j \mid \left( F_i \right)_j \neq 0\}$.
(Note that $\Tor_i^R(\Bbbk, J) = F_i \otimes_R \Bbbk$.)
Hence we see that $t_{i+1} \geq t_i + 1$ for all $i \geq 0$. 
Thus it suffices to prove the proposition for $i=2$.

Write $t = t_2$. 
Write $V = R_1$.
Consider the Koszul complex
\[
K_\bullet :
0 \rightarrow \bigwedge^{n+1}V \otimes_\Bbbk R(-n-1)
\rightarrow \cdots \to
\bigwedge^{i}V \otimes_\Bbbk R(-i)
\rightarrow \cdots \to
V \otimes_\Bbbk R(-1) \rightarrow R \to 0
\]
which is a minimal free resolution of $\Bbbk = R/R_+$.
Then $\Tor_i^R(\Bbbk,J) = \homology_i(K_\bullet \otimes_R J)$.
From~\cite[\S 1]{GreenTrieste89}, we see that
\[
\Tor_2^R(\Bbbk, J)_{t}  = 
\homology\left(
\bigwedge^3 V \otimes_\Bbbk J_{t-3} \to
\bigwedge^2 V \otimes_\Bbbk J_{t-2} \to 
V \otimes_\Bbbk J_{t-1}
\right).
\]
Now assume, by way of contradiction, that $t < \delta_3+2$. 
Let $J'$ be the ideal
generated by the two generators of $J$ in degrees $\delta_1$ and 
$\delta_2$.
Then $J_j = J'_j$ for all $j \leq t-2$ and $J'_{t-1} \subseteq J_{t-1}$.
Looking at the maps in the above complex (cf.~\cite[\S 1]{GreenTrieste89})
we see that the diagram
\[
\xymatrix{
\bigwedge^3 V \otimes_\Bbbk J'_{t-3} \ar[r] \ar@{=}[d] &
\bigwedge^2 V \otimes_\Bbbk J'_{t-2} \ar[r]^{\alpha'} \ar@{=}[d] &
V \otimes_\Bbbk J'_{t-1} \ar@{^(->}[d]
\\
\bigwedge^3 V \otimes_\Bbbk J_{t-3} \ar[r] &
\bigwedge^2 V \otimes_\Bbbk J_{t-2} \ar[r]^\alpha &
V \otimes_\Bbbk J_{t-1}
}
\]
commutes and that $\ker \alpha' = \ker \alpha$.
Hence $\Tor_2^R(\Bbbk, J)_{t}  = \Tor_2^R(\Bbbk, J')_{t}$.
Since $J'$ is generated by two elements, $\projdim_R J' \leq 1$, so 
$\Tor_2^R(\Bbbk, J')_{t}=0$, yielding the contradiction.
\end{proof}

\begin{proposition}
\label{proposition:socM1Z}
With notation as in Lemma~\ref{lemma:tid3}, assume that $J$ is
saturated. Let $Z \subseteq \projective^n$ be the subscheme defined by $J$.
Then $(\socle(M_1(Z)))_j = 0$ for all $j < \delta_3-2$.
\end{proposition}

\begin{proof}
Write $\calI_Z$ for the ideal sheaf defined by $J$.
Consider the sheafification $\calK_\bullet$ of the Koszul complex
$K_\bullet$ from the proof of Lemma~\ref{lemma:tid3},
\[
\calK_\bullet\; : \;
0 \rightarrow \bigwedge^{n+1}V \otimes_\Bbbk 
\strSh_{\projective^n}(-n-1) \rightarrow  
\bigwedge^{n}V \otimes_\Bbbk \strSh_{\projective^n}(-n) \rightarrow 
\cdots \rightarrow
V \otimes_\Bbbk \strSh_{\projective^n}(-1) \rightarrow 
\strSh_{\projective^n} \rightarrow  0,
\]
which is a locally free resolution of the $0$ sheaf on ${\projective^n}$.
Hence $\calK_\bullet \otimes_{\strSh_{\projective^n}} \calI_Z$ is exact.
We now describe the hyper-cohomology spectral sequence
(see, e.g.,~\cite[\href{https://stacks.math.columbia.edu/tag/015J}{Tag
015J}]{StacksProject})
for $\Gamma(\calK_\bullet \otimes_{\strSh_{\projective^n}}
\calI_Z(j+n+1))$. Since we need information about some maps, we describe
the steps up to page-1 of the spectral sequence.

We place the exact complex 
$\calK_\bullet \otimes_{\strSh_{\projective^n}} \calI_Z(j+n+1)$ on the
horizontal axis (arrows going to the right)
and consider a Cartan-Eilenberg injective resolution $I^{\bullet, \bullet}$
(arrows going right and up). Then $\Gamma(I^{\bullet, \bullet})$ is a
double complex in the upper half-plane, bounded on the left and on the
right; this is the $E_0$ page. We now take the homology of the vertical
maps, to get the following $E_1$ page
(with $\displaystyle \bigwedge^i V$ abbreviated as $\bigwedge^i$):
\[
\xymatrix@R=1em{%
\cdots && \ddots & \cdots \\
\bigwedge^{n+1} \otimes_\Bbbk \homology^1(\calI_Z(j)) \ar[r]^-\alpha & 
\bigwedge^{n} \otimes_\Bbbk \homology^1(\calI_Z(j+1))
 \ar[r] & 
\\
\bigwedge^{n+1} \otimes_\Bbbk J_j \ar[r] & 
\bigwedge^{n} \otimes_\Bbbk J_{j+1} \ar[r]^\beta  & 
\bigwedge^{n-1} \otimes_\Bbbk J_{j+2} \ar[r]^\gamma & 
\bigwedge^{n-2} \otimes_\Bbbk J_{j+3} \ar[r]&.
}
\]($J = \homology^0_*(\calI_Z)$ since $J$ is saturated.)
In the Cartan-Eilenberg resolution $I^{\bullet, \bullet}$, the map from the
left-most column to the next one lifts the map induced by 
\[
\xymatrix@C=3cm{
\bigwedge^{n+1}V \otimes_\Bbbk \strSh_{\projective^n}(-n-1) 
\ar[r]^{
\begin{bmatrix}
x_n \\ -x_{n-1} \\ \vdots \\ (-1)^n x_0
\end{bmatrix}
} & 
\bigwedge^{n}V \otimes_\Bbbk \strSh_{\projective^n}(-n) 
}.
\]
Therefore the map $\alpha$ on the $E_1$ page above is 
\[
\xymatrix@C=3cm{
M_1(Z)_j \ar[r]^{
\begin{bmatrix}
x_n \\ -x_{n-1} \\ \vdots \\ (-1)^n x_0
\end{bmatrix}
} & 
\left( M_1(Z)_{j+1} \right)^{\oplus n+1}
}.
\]
Hence $(\socle(M_1(Z)))_j = \ker \alpha$.
Since the spectral sequence abuts to $0$, we see that 
$\ker \alpha = (\ker \gamma)/(\image \beta)$.
On the other hand the last row of the above page is the strand
corresponding to degree $j+n+1$ of the complex 
$K_\bullet \otimes_R J$.
We conclude that
\[
(\socle(M_1(Z)))_j = \ker \alpha
= (\ker \gamma)/(\image \beta)
= \Tor_{n-1}^R(\Bbbk,J)_{j+n+1}.
\]
Therefore 
\[
\min \{ j \mid (\socle(M_1(Z)))_j \neq 0 \}
= t_{n-1} - (n+1) \geq \delta_3-2,
\]
by Lemma~\ref{lemma:tid3}.
\end{proof}

\begin{lemma}
\label{lemma:RedNoSecNeg}
Let $Z \subseteq \projective^n$ be a closed reduced subscheme without
zero-dimensional components.
Then $\homology^1(\calI_Z(j)) = 0$ for all $j < 0$.
\end{lemma}

\begin{proof}
It suffices to show that 
$\homology^0(\strSh_Z(j)) = 0$ for all $j < 0$.
Consider the ring $\Gamma_*(\strSh_Z)$.
Note that 
$\displaystyle \left(\Gamma_*(\strSh_Z)\right)_j = 0$ for all $j \ll 0$. Hence 
$\displaystyle \bigoplus_{j<0} \left(\Gamma_*(\strSh_Z)\right)_j$ is
nilpotent.
Let $j<0$. Then $(R_{- j}) \left(\Gamma_*(\strSh_Z)\right)_j \subseteq 
\homology^0(\strSh_Z)$, which is a reduced ring, so 
$(R_{- j}) \left(\Gamma_*(\strSh_Z)\right)_j=0$. Hence
$\displaystyle \bigoplus_{j<0} \left(\Gamma_*(\strSh_Z)\right)_j$ is
$R_+$-torsion. However, 
$\homology^0_{R_+}(\Gamma_*(\strSh_Z)) = 0$, which we see as follows:
By~\cite[Chapter~II, Proposition~5.15]{HartAG}, 
$\widetilde{\Gamma_*(\strSh_Z)} = \strSh_Z$; hence in the four-term exact
sequence (see, e.g.,~\cite[Theorem~A4.1]{eiscommalg})
\[
0 \to 
\homology^0_{R_+}(\Gamma_*(\strSh_Z))
\to
\Gamma_*(\strSh_Z)
\to
\bigoplus_{\nu \in \ints}\homology^0(\strSh_Z(\nu))
\to
\homology^1_{R_+}(\Gamma_*(\strSh_Z)) \to 0
\]
the middle map is an isomorphism.
\end{proof}

If $S$ is a noetherian ring and $M$ a finitely generated $S$-module, we say
that $M$ \define{satisfies the condition $(S_k)$} if for all $\frakp \in
\Spec S$, 
\[
\depth_{S_\frakp}(M_\frakp) \geq \min \{k, \dim S_\frakp\}.
\]
We need the following characterization of the conditions $(S_k)$;
see~\cite[\S 1.2]{SchenzelUseLC1998} for a proof.

\begin{proposition}
\label{proposition:SkAndExt}
Let $S$ be a regular local ring and $M$ a finitely generated $S$-module.
Let $k \geq 1$.
Then $M$ satisfies $(S_k)$ if and only if for all $i \geq 1$, $\codim_S
\Ext_S^i(M,S ) \geq i+k$.
\end{proposition}

\begin{lemma}
\label{lemma:S2UFDfree}
Let $S$ be a noetherian UFD and $M$ a finitely generated rank-one
$S$-module satisfying $(S_2)$. 
Then $M$ is free.
\end{lemma}

\begin{proof}
Write $\Frac(S)$ for the fraction field of $S$.
Since $M$ satisfies $(S_2)$, it is torsion-free. Hence the natural
$S$-module map $M \to M  \otimes_S \Frac(S) \simeq \Frac(S)$ is injective. 
Let $\Lambda \subseteq \Frac(S)$ be a finite set of generators of the image
of $M$ inside $\Frac(S)$. Multiplying by the product of the denominators of
the elements of $\Lambda$, we get an $S$-module embedding of $M$ into $S$;
write $I$ for the image of this map.
Then we have an exact sequence
\[
0 \to M \to S \to S/I \to 0.
\]

Let $\frakp \in \Ass (S/I)$. Since $I \neq 0$, $\frakp \neq 0$.
Hence by applying the `depth formula'~\cite[Proposition~1.2.9]{BrHe:CM}
to the localization of the above sequence at $\frakp$, we see that
$\depth_{S_\frakp } M_\frakp = 1$, so $\height \frakp = 1$ (by the $(S_2)$
property). In other words, $I$ is an unmixed height-one ideal. Since $S$ is
a UFD, $I$ is principal (using, e.g.,~\cite[Corollary~10.6]{eiscommalg}).
Therefore $M$ is free.
\end{proof}

Let $S$ be a $d$-dimensional $\naturals$-graded $\Bbbk$-algebra, with
homogeneous maximal ideal $R_+$. The \define{$a$-invariant} of $S$, denoted
by $a(S)$ is
\[
\max\{j \mid \homology^d_{S_+}(S)_j \neq 0 \}.
\]

\begin{proposition}
\label{proposition:aInvDim}
Let $S$ be a standard-graded $\Bbbk$-algebra.
Then $a(S) \geq -\dim S$.
\end{proposition}

\begin{proof}
Write $d = \dim S$.
Since the base-change of $S$ along a field extension $\Bbbk \to \Bbbk'$ is
faithfully flat, we may assume that $\Bbbk$ is infinite.
Hence there exists a Noether normalization $A$ that is a standard-graded
polynomial ring over $\Bbbk$, generated by $d$ elements of $S_1$; see,
e.g.,~\cite[Theorem~13.3]{eiscommalg}.
Since a regular ring in equi-characteristic is a direct summand of every
finite extension~\cite{HochsterContracted1973} we see that the induced map
\[
\homology^d_{A_+}(A) \to \homology^d_{S_+}(S)
\]
splits, so $a(S) \geq a(A) = - \dim A = -\dim S$.
\end{proof}

\begin{discussionbox}
\label{discussionbox:engheta}
(See~\cite[Proposition~11]{EnghetaPDUnmThreeCubics2007}.)
Let $\frakp := (x_0, x_1) \subseteq R$. Let $J$ be a homogeneous
$\frakp$-primary ideal such that the multiplicity $e(R/J) = 2$.
Then, after a homogeneous change of coordinates, if necessary, 
$J = (x_0, x_1^2)$ or $J =  \frakp^2 + (g_0x_0 + g_1x_1)$
where $g_0$ and $g_1$ are homogeneous polynomials of (the same)
positive degree such that $x_0, x_1, g_0, g_1$ is a regular sequence.
In particular, $\projdim_R R/J \leq 3$.
\end{discussionbox}

For a homogeneous $R$-ideal $J$, we write $[J]_d$ for the
$\Bbbk$-vector space of the homogeneous elements of $J$ of degree $d$.

\begin{proposition}
\label{proposition:mult3}
Let $\frakp := (x_0, x_1) \subseteq R$. Let $J$ be a homogeneous
$\frakp$-primary ideal such that the multiplicity $e(R/J) = 3$ 
and $\rank_\Bbbk[J]_2 \geq 3$.
Then $\projdim_R R/J \leq 3$.
\end{proposition}

\begin{proof}
If $\frakp^2 \subseteq J$, then $J = \frakp^2$, in which case $\projdim_R R/J 
= 2$.
Otherwise, $JR_\frakp \not \subset (\frakp R_\frakp)^2$; therefore
$\frakp^3 \subseteq J$ and $J \not \subset \frakp^2$.
Hence $[J]_2 \cap \frakp^2$ has dimension at most $2$, so
there exist linear forms $l_0, l_1$ such that $h := l_0x_0  + l_1x_1
\in [J]_2 \subseteq J$ with $l_0 \not \in \frakp$ or $l_1 \not \in \frakp$.
Write $J_1 = \frakp^3 + (l_0x_0  + l_1x_1)$.
Note that $J$ is the $\frakp$-primary component of $J_1$.

Let $R' := \Bbbk[ x_0, x_1, l_0, l_1]$. 
For ideals of $R$, write $(-)' = - \cap R'$.
Then $J_1 = J_1' R$. Write $J_1' = \fraka \cap \frakb$ 
where $\fraka$ is $\frakp'$-primary and 
$\frakb$ is the intersection of the components associated to the remaining
associated prime ideals of $J_1'$.
Since $R' \to R$ is faithfully flat, 
$J_1 = J_1'R = \fraka R \cap \frakb R$.
Hence $J = J_1R_\frakp \cap R = \fraka R_\frakp \cap R$.
(Note that there exists $f \in R' \minus \frakp' \subseteq R \minus \frakp$ 
such that $f^m \in \frakb$ for all $m \gg 0$, so $\frakb R_\frakp =
R_\frakp$.) 

Since $R$ is faithfully flat over $R'$ and $\fraka$ is $\frakp'$-primary,
we know (from~\cite[(9.B)]{MatsCA80}) that
\[
\Ass_R (R/\fraka R) = \Ass_R (R/\frakp' R) = \{\frakp \}.
\]
Hence $J = \fraka R_\frakp \cap R =\fraka R$, from which it follows that
$\projdim_R R/J = \projdim_{R'} R'/\fraka < \depth R' = 4$, since 
the homogeneous maximal ideal of $R'$ is not associated to 
$R'/\fraka$.
\end{proof}

\section{A construction}
\numberwithin{equation}{section}
\label{section:construction}

As in the Introduction, let $X \subseteq \projective^n$, $n \geq 4$, 
be a subcanonical
local complete intersection scheme of codimension two with $\omega_X =
\strSh_X(a_X)$.
Recall that we have an exact sequence~\eqref{equation:sesCalE}:
\[
0 \to \strSh_{\projective^n}(-(a_X+n+1)) \to \calE \to \calI_X \to 0.
\]
Applying $\Gamma_*(-)$ we get an exact sequence of $R$-modules (with $M :=
\Gamma_*(\calE)$):
\begin{equation}
\label{equation:MIX}
0 \to R(-(a_X+n+1)) \to M \stackrel{\beta}\to I_X \to 0.
\end{equation}
Note that $M$ is generated (not necessarily minimally) in degrees $d_i, 1
\leq i \leq m$ and $a_X+n+1$.

\begin{lemma}
\label{lemma:Mdual}
$M^* \simeq M(a_X +n + 1)$.
\end{lemma}

\begin{proof}
Since $\det \calE = \strSh_{\projective^n}(-(a_X+n+1))$, we see that
$\calE^* \simeq \calE(a_X+n+1)$ and, hence, that
the sheafification of $M^*$ is $\calE(a_X+n+1)$.
Since $M^*$ satisfies $(S_2)$, $M^* = \Gamma_*(\calE(a_X+n+1))$. This
follows from the four-term exact sequence connecting local cohomology and
$\Gamma_*$; see, e.g.,~\cite[Theorem~A4.1]{eiscommalg}.
\end{proof}

Now suppose that $X$ is not a complete intersection.
We construct a local complete intersection $Z \subseteq
{\projective^n}$ of codimension two as follows.
Let $\xi \in M$ be a homogeneous element of degree $d_1$ that maps to a
non-zero element of $I_X$ under the map $\beta$ from~\eqref{equation:MIX}.
(The image of $\xi$ is part of a minimal homogeneous generating set of
$I_X$.)

\begin{lemma}
\label{lemma:Mmodxi}
With notation as above, assume that $2d_1 \leq a_X +n + 2$. 
Further, write $i$ for the inclusion map $R\xi \to M$.
Let $\displaystyle J = \Ann_R\left(\coker \left(M^*
\stackrel{i^*}{\to}R(d_1)\right)\right)$.
Then 
\begin{enumerate}

\item
\label{lemma:Mmodxi:S1}
$M/R\xi$ satisfies $(S_1)$.

\item
\label{lemma:Mmodxi:NdblDual}
$(M/R\xi)^{**} = R(d_1-(a_X + n + 1))$.

\item
\label{lemma:Mmodxi:Junm}
$J$ is a unmixed height-two ideal and $J_\frakp$ is a complete intersection
for all $\frakp \neq R_+$.

\end{enumerate}
\end{lemma}

\begin{proof}
Write $N=M/R\xi$. Then we have an exact sequence
\[
0 \to N^* \to M^* \stackrel{i^*}\to R(d_1) \to \Ext^1_R(N,R) \to 
\Ext^1_R(M,R) \to 0
\]
and isomorphisms $\Ext^i_R(N,R) \stackrel\sim\to \Ext^i_R(M,R)$ for all $i
\geq 2$.
By Proposition~\ref{proposition:SkAndExt}, we need to show that 
$\codim_R(\Ext_R^i(N,R)) \geq i+1$.
Since $M$ satisfies $(S_2)$ and 
$\Ext_R^i(M,R )$ is a finite-length $R$-module for every $i>0$, 
it remains to show that $\codim_R(\Ext_R^1(N,R)) \geq 2$.
Since $\displaystyle\Supp \left(\Ext_R^1(N,R) \right ) = \Supp (R/J)$, we
need to show that $\height J \geq 2$.
It is immediate that $\height J >0$.
Note also that $N^*$ is free, since it is a rank-one module satisfying
$(S_2)$ (Lemma \ref{lemma:S2UFDfree}).

\eqref{lemma:Mmodxi:S1}. First assume that $J$ contains a linear form $l$.
If $\height J = 1$, then $J = (l)$, from which it would follow that
$\image i^*$ is a free $R$-module and that $M^*$ and $M$ are free.
This is a contradiction, so $\height J \geq 2$, which
proves~\eqref{lemma:Mmodxi:S1}.
Hence we may assume that $J$ does not contain a linear form, or,
equivalently, that $\image i^*$ is zero in degree $\leq -d_1+1$.

Since $M^* \simeq M(a_X +n + 1)$ (Lemma~\ref{lemma:Mdual}),
it is generated in degrees
$d_i - (a_X + n + 1)$, $1 \leq i \leq m$ and $0$.
Since $2d_1 \leq a_X + n + 2$, we see that
$d_1 - (a_X + n + 1) \leq -d_1+1$.
Therefore $N^*$ maps on to a minimal generator $\eta$ of $M^*$ degree 
$d_1 - (a_X + n + 1)$ and the remaining minimal generators of $M^*$ 
have degree $\geq -d_1+2$. 
It follows that 
$M^*/R\eta = \image i^* \simeq N(a_X+n+1)$, so $N$ satisfies $(S_1)$.

\eqref{lemma:Mmodxi:NdblDual}:
$N^* = R(a_X+n+1-d_1)$, by the above argument. 

\eqref{lemma:Mmodxi:Junm}:
By~\eqref{lemma:Mmodxi:S1}, $\height J \geq 2$.
Let $\frakp \in \Ass R/J$. 
Since $\depth M^* \geq 2$ and $\depth R \geq 3$, 
we see that $\frakp \neq R_+$.
Hence we have a free resolution
\[
0 \to N_\frakp^* \to M_\frakp^* \to R_\frakp \to  0
\]
of $(R/J)_{\frakp}$ over $R_\frakp$. Since $\depth_{R_\frakp } 
(R/J)_{\frakp} = 0$, we see from the Auslander-Buchsbaum formula that
$\height \frakp \leq 2$. Since $\height J \geq 2$, we conclude that
$\height \frakp = 2$. Hence $J$ is a unmixed height-two ideal.
Further for all $\frakp \neq R_+$, $M_\frakp^*$ is a rank-two free
$R_\frakp$-module, so $J_\frakp$ is a complete intersection.
\end{proof}

\begin{definition}
\label{definition:Z}
Let $Z = \Proj(R/J)$, where $J$ is the ideal from
Lemma~\ref{lemma:Mmodxi}. Write $\calI_Z$ for its ideal sheaf.
\end{definition}

\section{Proofs}
\label{section:proofs}

We start with some observations about the scheme $Z$ constructed in the
previous section.

\begin{proposition}
\label{proposition:Zsubc}
The scheme $Z$ defined in Definition~\ref{definition:Z}
is a subcanonical codimension-two local complete intersection subscheme 
of $\projective^n$, with $\omega_Z = \strSh_Z(2d_1-a_X-2n-2)$.
\end{proposition}

\begin{proof}
The exact sequence in the proof of Lemma~\ref{lemma:Mmodxi} gives the
following exact sequence (still in the notation of~Lemma~\ref{lemma:Mmodxi}
and its proof):
\begin{equation}
\label{equation:NMJd1}
0 \to N^* \to M^* \to J(d_1) \to 0.
\end{equation}
This gives (see the proof of
Lemma~\ref{lemma:Mmodxi}\eqref{lemma:Mmodxi:NdblDual}) the exact sequence
\begin{equation}
\label{equation:sesForIZ}
0 \to \strSh_{\projective^n}(a_X + n + 1 -2d_1) \to \calE(a_X + n + 1-d_1)
\to \calI_Z \to 0.
\end{equation}
Now use Proposition~\ref{proposition:lcisubcan}.
\end{proof}

\begin{proof}[Proof of Theorem~\protect{\ref{theorem:minHOne}}]
By way of contradiction, 
assume that $X$ is not a complete intersection.  
Let $Z$ be the l.c.i subscheme of $\projective^n$ defined in
Definition~\ref{definition:Z}.
Let $\epsilon = \min \{ j \mid \homology^0(\strSh_Z(j)) \neq 0\}$.

We first argue that $\epsilon \geq  2d_1 - a_X -n -2$.
Since $2d_1 \leq a_X +n +2$, we may assume that $\epsilon < 0$.
Therefore $\homology^1(\calI_Z(\epsilon)) \neq 0$.
From~\eqref{equation:sesForIZ} we see that 
$\homology^1(\calE(a_X + n + 1-d_1 + \epsilon)) \neq 0$, and, consequently,
that
$\homology^1(\calI_X(a_X + n + 1-d_1 + \epsilon)) \neq 0$.
Hence $\epsilon \geq d_1-1 - (a_X+n+1-d_1) = 2d_1 - a_X -n -2$.

Now we estimate $a(R/J)$.
\begin{align*}
a(R/J) & = \max \{j \mid \homology^{n-1}_{R_+}(R/J)_j \neq 0\}
\\
& = \max \{j \mid \homology^{n-2}(\strSh_Z(j)) \neq 0\}
\\
& = - \min \{j \mid \homology^{0}(\omega_Z(j)) \neq 0\} \qquad 
\text{(by Serre duality)}
\\
& = - \min \{j \mid \homology^{0}(\strSh_Z(2d_1-a_X-2n-2+j)) \neq 0\}
\\
& = 2d_1-a_X-2n-2- \epsilon
\\
& \leq -n.
\end{align*}
This yields a contradiction, since, by
Proposition~\ref{proposition:aInvDim},
$a(R/J) \geq -n+1$.
Hence $X$ is a complete intersection.
\end{proof}

\begin{proposition}
\label{proposition:Z}
Let $Z$ be the subscheme of $\projective^n$ defined in
Definition~\ref{definition:Z}.
Then 
\begin{enumerate}

\item
\label{proposition:Z:notred}
$Z$ is not reduced.

\item
\label{proposition:Z:M}
$M_1(Z) = M_1(X)(a_X+n+1-d_1)$.

\item
\label{proposition:Z:notCI}
$Z$ is not a complete intersection.

\item
\label{proposition:Z:depthRJ}
$\depth R/J = 1$.

\item
\label{proposition:Z:JnoLin}
$J$ does not contain any linear form.

\end{enumerate}
\end{proposition}

\begin{proof}
\eqref{proposition:Z:notred}:
By way of contradiction, suppose that $Z$ is reduced. Then
$\homology^1(\calI_Z(j)) = 0$ for all $j<0$
(Lemma~\ref{lemma:RedNoSecNeg}).
Therefore
\[
\min\{j \mid \homology^1(\projective^n, \calI_X(j ) ) \neq 0 \} 
\geq (a_X+n+1)-d_1 \geq d_1-1.
\]
Then Theorem~\ref{theorem:minHOne} yields a contradiction.

\eqref{proposition:Z:M}:
Note that $M_1(X)  = \homology^1_*(\calE)$ and 
that $M_1(Z)  = \homology^1_*(\calE(a_X+n+1-d_1))$.

\eqref{proposition:Z:notCI}:
If $Z$ were a complete intersection, then 
$\calE(a_X+n+1-d_1)$ would split (Proposition~\ref{proposition:lcisubcan})
from which it would follow that $X$ is a complete intersection.

\eqref{proposition:Z:depthRJ}:
By way of contradiction, assume that $\depth R/J \geq 2$. (Since $J$ is
unmixed, $\depth R/J >0$.) Then $\depth M^* \geq 3$. For all 
$\frakp \in \Spec R$, $\frakp \neq R_+$, $M^*_\frakp$ is a free
$R_\frakp$-module. Hence $M^*$ satisfies $(S_3)$. A rank-$2$ module
satisfying $(S_3)$ is free~\cite[Theorem~1.1]{EvansGriffithSyzProb1981}.
This implies that $Z$ is a complete intersection,
contradicting~\eqref{proposition:Z:notCI}.

\eqref{proposition:Z:JnoLin}: Since $J$ is a height-two unmixed ideal, if
it contains a linear form, it would be
a complete intersection, contradicting~\eqref{proposition:Z:notCI}.
\end{proof}

In the next two lemmas, $\Min(R/J)$ is the set of minimal prime ideals of
$R$ containing $J$.
\begin{lemma}
\label{lemma:3quadMult}
Suppose that  $\rank_\Bbbk[J]_2 \geq 3$.
Then the multiplicity $e(R/J) \leq 3$.
\end{lemma}

\begin{proof}
Let $f_1, f_2, f_3$ be linearly independent quadratic polynomials in $J$.
If two of them say, without loss of generality, $f_1, f_2$, 
form a regular sequence, then the multiplicity $e(R/(f_1, f_2)) = 4$; 
since $J$ is unmixed and $J \neq (f_1, f_2)$, $e(R/J) \leq 3$. 

Otherwise, none of $f_1, f_2, f_3$ is irreducible.
There cannot exist a linear form $l$ that divides $f_i$ for all $1 \leq i
\leq 3$; for, otherwise, the height three ideal
$\displaystyle \left( \frac{f_1}{l}, \frac{f_2}{l},
\frac{f_3}{l} \right)$ would be subset of an associated prime ideal of
$R/J$, a contradiction of the fact that $J$ is a height-two unmixed ideal.
Therefore there exist pairwise linearly independent linear forms $l_1, l_2,
l_3$ such that 
$(f_1, f_2 ) \subseteq (l_1)$, $(f_2, f_3 ) \subseteq (l_2)$ and
$(f_1, f_3 ) \subseteq (l_3)$. 
Then, without loss of generality, $f_1 = l_1l_3$, $f_2 = l_1l_2$, $f_3 =
l_2l_3$.
If $l_1, l_2, l_3$ are linearly independent, then 
\[
(f_1, f_2, f_3) = (l_1, l_2) \cap (l_2, l_3) \cap (l_1, l_3).
\]
Note that 
\[
\Ass(R/J) = \Min(R/J) \subseteq \Min(R/(f_1, f_2, f_3)) = 
\{(l_1, l_2) , (l_2, l_3) , (l_1, l_3) \}
\]
and $\lambda_{R_\frakp}(R_\frakp/JR_\frakp) \leq 
\lambda_{R_\frakp}(R_\frakp/(f_1, f_2, f_3)R_\frakp)$
for every $\frakp \in \Ass(R/J)$. This would imply that $J$ is reduced.
Therefore we conclude that 
$l_1, l_2, l_3$ span a two-dimensional vector-space
and $(f_1, f_2, f_3) = (l_1, l_2)^2$. Since $\dim R/J = \dim R/(l_1,
l_2)^2$, we again see that $e(R/J) \leq 3$.
\end{proof}

\begin{lemma}
\label{lemma:3quadDepth}
Suppose that  $\rank_\Bbbk[J]_2 \geq 3$.
Then $\depth (R/J) \geq n-2$.
\end{lemma}

\begin{proof}
By Lemma~\ref{lemma:3quadMult}, $e(R/J ) \leq 3$.
Using the additivity formula of multiplicity
(cf., e.g.,~\cite[Exercise~12.11]{eiscommalg})
and the fact that $Z$ is not reduced
(Proposition~\ref{proposition:Z}\eqref{proposition:Z:notred}), we see that
$|\Min(R/J)| \leq 2$.

\underline{$|\Min(R/J)| = 1$}: 
Write $\Min(R/J) = \{\frakp\}$.
Since $J$ is not reduced, $\lambda_{R_\frakp}(R/J)_{\frakp} \geq 2$.
Hence $2 e(R/\frakp) \leq e(R/J) \leq 3$,
so $e(R/\frakp) = 1$.
Hence, without loss of generality, $\frakp = (x_0, x_1)$.
If $e(R/J) = 2$, then $J$ is described in 
Discussion~\ref{discussionbox:engheta}.
Since $J$ does not contain a linear form
(Proposition~\ref{proposition:Z}\eqref{proposition:Z:JnoLin})
we see that $J = \frakp^2 + (g_0x_0+g_1x_1)$ 
and that $\projdim_R R/J =3$.
On the other hand, if $e(R/J) = 3$, then 
$\projdim_R R/J \leq 3$ by Proposition~\ref{proposition:mult3}.
Now use the Auslander-Buchsbaum formula.

\underline{$|\Min(R/J)| = 2$}: 
Write $\Min(R/J) = \{\frakp_1, \frakp_2\}$.
Since $J$ is not reduced, $2 \leq e(R/\frakp_1) + e(R/\frakp_2 ) 
< e(R/J) \leq 3$,
so $e(R/\frakp_1) = e(R/\frakp_2 )  = 1$.
Without loss of generality, $J = J_1 \cap \frakp_2$ with 
$J_1$ associated to $\frakp_1$ and $e(R/J_1) = 2$. 
Without loss of generality, $\frakp_1 = (x_0, x_1)$; the possible values of
$J_1$ is described in Discussion~\ref{discussionbox:engheta}.
Write $\frakp_2 = (y_2, y_3)$ for some linear forms $y_2, y_3$.

Suppose that $J_1 = (x_0, x_1^2)$.
If $\height (x_0, x_1, y_2, y_3) = 4$, then $[J]_2 = \langle x_0y_2,
x_0y_3\rangle$, so, without loss of
generality, $y_2 \in (x_0, x_1) \minus (x_0)$ and $y_3 \not \in (x_0,
x_1)$.
Now consider the exact sequence
\[
0 \to R/J \to 
\begin{matrix}
R/(x_0, x_1^2) \\ \oplus \\ R/(y_2, y_3) 
\end{matrix}
\to R/(x_0, x_1^2, y_2, y_3) \to 0.
\]
Note that $(x_0, x_1^2, y_2, y_3)=(x_0, y_2, y_3)$ is generated by a
regular sequence of length $3$. Hence
\[
\depth R/(x_0, x_1^2, y_2, y_3)  = \dim R/(x_0, x_1^2, y_2, y_3)= n-2.
\]
Therefore $\depth R/J = \dim R/J = n-1$.

Hence suppose that $J_1 =  (x_0^2, x_0x_1, x_1^2, g_0x_0 + g_1x_1)$ with
$g_0, g_1$ homogeneous such that $x_0, x_1, g_0, g_1$ is a regular
sequence. If the $g_i$ are not linear forms, then 
$\rank_\Bbbk [J]_2<3$, so the $g_i$ are linear forms.

Consider the intersection $[J]_2 \cap \langle x_0^2, x_0x_1 \rangle$
(inside $[J_1]_2$). This is not zero-dimensional, so $\frakp_2$ contains a
linear form inside $\frakp_1$.
Without loss of generality $y_2 \in
\frakp_1$ and $y_3 \not \in \frakp_1$. Therefore 
$\frakp_2 \cap \langle x_0^2, x_0x_1, x_1^2 \rangle$ is $2$-dimensional.
On the other hand, $[J_1]_2 \cap \frakp_2$ has dimension at least $3$,
so there exist $a \in \Bbbk \minus \{0 \}$ and $q_0 \in \Bbbk \langle x_0^2,
x_0x_1, x_1^2 \rangle$ such that 
\[
q := a(g_0x_0 + g_1x_1) + q_0 \in (y_2, y_3 ).
\]
Let $y_1 \in \frakp_1$ be a linear form such that $\frakp_1 = (y_1, y_2)$.
Then $J_1 = (y_1, y_2)^2 + (q)$ and, therefore, 
$J_1 + \frakp_2 =  (y_1^2, y_2, y_3)$, which is generated by a regular
sequence of length $3$. 
It now follows from the exact sequence 
\[
0 \to R/J \to 
\begin{matrix}
R/J_1 \\ \oplus \\ R/(y_2, y_3) 
\end{matrix}
\to R/(J_1 + \frakp_2) \to 0
\]
that $\depth R/J \geq n-2$.
\end{proof}

\begin{discussionbox}
We now observe that $J$ has minimal generators in degrees $d_1 + d_i - (a_X
+ n + 1)$ for $2 \leq i \leq m$, and a possibly non-minimal generator in
degree $d_1$. First, note that $M$ has minimal generators in degrees $d_i$
for $1 \leq i \leq m$, and a possibly non-minimal generator in degree $(a_X
+ n + 1)$. Hence, by Lemma~\ref{lemma:Mdual}, $M^*$ has minimal generators
in degrees $d_i - (a_X + n + 1)$ for $1 \leq i \leq m$, and a possibly
non-minimal generator in degree $0$. We see from~\eqref{equation:NMJd1}
that the module $J(d_1)$ has minimal generators in degrees $d_i - (a_X + n
+ 1)$ for $2 \leq i \leq m$, and a possibly non-minimal generator in degree
$0$. This proves the above statement about the generators of $J$.

We can further see the following: 
By Proposition~\ref{proposition:Z}\eqref{proposition:Z:JnoLin}, $d_1 + d_i
- (a_X + n + 1) \geq 2$ for $2 \leq i \leq m$.
This can be slightly improved. From
Proposition~\ref{proposition:Z}\eqref{proposition:Z:depthRJ} and
Lemma~\ref{lemma:3quadDepth}, we see that $\rank_\Bbbk [J]_2 \leq 2$. Hence
$d_1 + d_i - (a_X + n + 1) \geq 3$ for $4 \leq i \leq m$.
On the other hand, we cannot decide whether 
$d_1 + d_3 - (a_X + n + 1) \geq 3$ or not, so we are unable to say which of
the quantities would be the maximum in the statement of 
Theorem~\ref{theorem:CIorNegSocZero} and Corollary~\ref{corollary:gap}.
\end{discussionbox}

\begin{proof}[Proof of Theorem~\protect{\ref{theorem:CIorNegSocZero}}]
Since $X$ is not a complete intersection,
let $Z$ be the l.c.i subscheme of $\projective^n$ defined in
Definition~\ref{definition:Z}.
Applying Proposition~\ref{proposition:socM1Z} to $X$ itself, we see that 
$(\socle(M_1(X)))_j = 0$ for all $j < d_3-2$; hence we need to show that 
$(\socle(M_1(X)))_j = 0$ for all $j < a_X+n+2-d_1$.

If $\rank_\Bbbk [J]_2 \geq 3$, then
by Lemma~\ref{lemma:3quadDepth}, $\depth R/J \geq n-2 \geq 2$, 
contradicting Proposition~\ref{proposition:Z}\eqref{proposition:Z:depthRJ}.
Therefore $\rank_\Bbbk [J]_2 \leq 2$.
Then in the notation of 
Lemma~\ref{lemma:tid3}, $\delta_3 \geq 3$.
By Proposition~\ref{proposition:socM1Z},
$(\socle(M_1(Z)))_j = 0$ for all $j \leq 0$, which implies that 
$(\socle(M_1(X)))_j = 0$ for all $j < a_X+n+2-d_1$
(Proposition~\ref{proposition:Z}\eqref{proposition:Z:M}).
\end{proof}

\begin{proof}[Proof of Corollary~\protect{\ref{corollary:3Buchs}}]
Note that $X$ is
subcanonical~\cite[Theorem~2.2(d)]{HartshorneVarSmallCodim1974}.
Write $\omega_X = \strSh_X(a_X )$.

We need to show that if
$2d_1 \leq a_X + n+1$ and $X$ is $3$-Buchsbaum, then $X$ is a complete
intersection. (For this, we do not require that $X$ is smooth, but only
that it satisfies the hypothesis of Theorems~\ref{theorem:minHOne}
and~\ref{theorem:CIorNegSocZero}. Similarly, we only need that 
the $R$-module $(R_+)^3 M_1(X)$ 
--- a submodule of the $R$-module $M_1(X)$ ---
is zero.)
By way of contradiction assume that $X$ is not a complete intersection.
Note that $(R_+)^2 M_1(X) \subseteq \socle(M_1(X))$.
By Theorem~\ref{theorem:CIorNegSocZero}, 
$(\socle(M_1(X)))_j = 0$ for all $j < a_X+n+2-d_1$.
Hence
\[
\min\{j \mid \homology^1(\projective^n, \calI_X(j)) \neq 0 \}
\geq a_X+n+2-d_1-2 \geq d_1-1,
\]
since $2d_1 \leq a_X+n+1$. Now apply Theorem~\ref{theorem:minHOne}.
\end{proof}

\begin{proof}[Proof of Corollary~\protect{\ref{corollary:gap}}]
Consider the set 
\[
K := \{k \in \ints \mid k < j \;\text{and}\; 
\homology^1(\calI_X(k)) \neq 0\}.
\]
If $K \neq \varnothing$, then let $k = \max K$. Then 
$(M_1(X))_k \subseteq \socle(M_1(X))$. 
Note that $k < \max \{ a_X+n+2-d_1, d_3-2\}$.
Hence Theorem~\ref{theorem:CIorNegSocZero} yields that $X$ is a complete
intersection.
On the other hand, if $K = \varnothing$, then apply
Theorem~\ref{theorem:minHOne}.
\end{proof}

A remark on the proof of Theorem~\ref{theorem:CIorNegSocZero}.
It uses the Evans-Griffith syzygy
theorem~\cite[Theorem~1.1]{EvansGriffithSyzProb1981} to see that $\depth
R/J =1$
(Proposition~\ref{proposition:Z}) and obtain a contradiction when we later
argued that $\depth R/J \geq n-2$. 
If $n \geq 5$, we can avoid the use 
of the Evans-Griffith syzygy theorem as follows. 
As earlier, we see that $\depth R/J \geq n-2$. 
Hence $\homology^i_{R_+}(J)=0$ for $j \leq n-2$; consequently,
$\homology^i_{R_+}(M)=0$ for $j \leq n-2$ (see the exact sequence in the
proof of Lemma~\ref{lemma:Mmodxi})
and $\homology^i_*(\calE)=0$ for all $1 \leq i \leq n-3$.
By~\eqref{equation:duality}, 
$\homology^i_*(\calE)=0$ for all $3 \leq i \leq n-1$.
In other words, $\homology^i_*(\calE)=0$ for all $1 \leq i \leq n-1$.
Now use the Horrocks' splitting 
criterion~\cite[Chapter~1, Theorem~2.3.1]{OSSvb2011}
to see that $X$ is a complete intersection.

\section{Comparison with some known results}
\label{section:compare}

In this section we compare the hypothesis of Theorem~\ref{theorem:minHOne}
with some earlier results on the splitting of rank-two vector bundles on
$\projective^n$. The hypothesis of Theorem~\ref{theorem:CIorNegSocZero} is
related to the Buchsbaum property; 
as we mentioned in the Introduction, it entails
Corollary~\ref{corollary:3Buchs}, which complements a result of Ellia and
Sarti~\cite[Proposition~13]{ElliaSartikBuchs1999}.

Adopt the notation of Theorem~\ref{theorem:minHOne}.
Consider the exact sequence
\begin{equation}
\label{equation:EIX}
0 \to \strSh_{\projective^n}(-(a_X+n+1)) \to \calE \to \calI_X \to 0
\end{equation}
with $\calE$ a rank-two vector bundle.
Then the Chern classes of $\calE$ are $c_1(\calE) = -(a_X+n+1)$ and
$c_2(\calE) = \deg X$~\cite[Chapter~1, 5.1.1]{OSSvb2011}.
Holme-Schneider \cite {HolmeSchnComputerAided1985} and
Holme~\cite {HolmeCodimTwo1989} prove the following sufficient
conditions for $\calE$ to split in terms of its Chern classes.
\begin{enumerate}

\item
$\calE$ is not stable and $c_2(\calE) < (n-1)
(c_1(\calE)-n+2)$~\cite[Theorem~4.2]{HolmeSchnComputerAided1985}.
This improved earlier results of Ran and of Peskine.

\item
$\calE$ is not stable and $c_2(\calE) \leq
(n-1)c_1(\calE)-n^2+3n$~\cite[Theorem~7.1]{HolmeCodimTwo1989}.
\end{enumerate}
(There are other papers that assume a bound on $c_2(\calE)$ in terms of
$n$, not involving $c_1(\calE)$.)

We give an example over $\projective^4$ below to see that the hypothesis of
Theorem~\ref{theorem:minHOne} does not imply the above hypotheses; hence we
believe that our results cannot be recovered by earlier results.

\begin{proposition}
\label{proposition:calF}
Let $\calF$ be a rank-two vector bundle on $\projective^4$ with Chern
classes $c_i(\calF)$, $i =1, 2$. 
Then the Euler characteristic of $\calF$ is
\[
\chi(\calF) = 
\frac 1{24} c_1^4 + \frac5{12} c_1^3 - \frac 16 c_1^2c_2 + 
\frac{35}{24}c_1^2 - \frac54c_1c_2 + \frac 1{12} c_2^2 + 
\frac{25}{12}c_1 - \frac{35}{12}c_2 +2,
\]
where $c_i = c_i(\calF)$. 
\end{proposition}

\begin{proof}
We use the illustration of the Hirzebruch-Riemann-Roch theorem by
Schwarzenberger in~\cite[\S 22.4, Appendix One]{HirzebruchTMAG1995}.
Let $\delta_i, i=1,2$ be the Chern roots
of $\calF$, i.e., $\displaystyle 1 + c_1(\calF) x + c_2(\calF) x^2
= (1 + \delta_1x ) (1 + \delta_2x )$.
Then
\[
\chi(\calF) = \binom{\delta_1 + 4}{4} + \binom{\delta_2 + 4}{4}.
\]
We rewrite this expression using \cite[Example 15.1.2]{FultonIntThy1998}.
Note that 
\[
\binom{x + 4}{4} = 
\frac{1} {24} ( x^4 + 10x^3 + 35 x^2 + 50x + 24).
\]
For $0 \leq k \leq 4$, write $p_k = \delta_1^k + \delta_2^k$.
Hence
\[
\chi(\calF) = 
\frac{1} {24} ( p_4 + 10p_3 + 35 p_2 + 50p_1 + 24p_0).
\]
The $c_i$ are the elementary symmetric functions in the $\delta_i$, so we
can express $p_k$ in terms of the $c_i$ as follows:
$p_0 = 2$,
$p_1 = c_1$,
$p_2 = c_1^2 - 2 c_2$,
$p_3 = c_1^3 - 3 c_1 c_2$ and
$p_4 = c_1^4 - 4 c_1^2 c_2 + 2 c_2^2$.
Substitute these into the above expression to complete the proof.
\end{proof}

\begin{proposition}
\label{proposition:calE}
Assume the hypothesis of Theorem~\ref{theorem:minHOne} and that $n=4$.
Let $\calE$ be as in~\eqref{equation:EIX}.
Write $c_i = c_i(\calE)$, $i=1,2$.
\begin{enumerate}

\item
If $2d_1=a_X+4$, then 
$(c_1^2 - 4 c_2 - 9)(c_1^2 - 4 c_2 - 1) \geq 0$.

\item
If $2d_1=a_X+5$, then 
$(c_1^2 - 4 c_2 - 4)(c_1^2 - 4 c_2 ) \geq 0$.

\item
If $2d_1=a_X+6$, then 
$(c_1^2 - 4 c_2 - 1)(c_1^2 - 4 c_2 + 15 ) \geq 0$.

\end{enumerate}

\end{proposition}

\begin{proof}
We first note that $2d_1 \leq a_X+6$ by the hypothesis of
Theorem~\ref{theorem:minHOne}.
Moreover, since $\homology^1(\calE(j)) = 0$ for all $j \leq d_1-2$, we see
that if $2d_1 \geq a_X+4$, then $\homology^3(\calE(j)) = 0$ for all $j \geq
d_1-2$ by~\eqref{equation:duality}, and, hence,
\[
\chi(\calE(d_1-2)) = \dim_\Bbbk \homology^2(\calE(d_1-2)) \geq 0.
\]
The expressions in $c_1$ and $c_2$ given in the statement of this
proposition are the factors of $\chi(\calE(d_1-2))$ (up to multiplication
by a positive integer); we will prove this using
Proposition~\ref{proposition:calF}.

From, e.g.,~\cite[Chapter~1, \S 1.2]{OSSvb2011}, we know that 
\begin{align*}
c_1(\calE(d_1-2)) & = c_1(\calE) + 2(d_1-2);\\
c_2(\calE(d_1-2)) & = c_2(\calE) + (d_1-2)c_1(\calE)  + (d_1-2)^2.
\end{align*}
We now use Proposition~\ref{proposition:calF} with $\calF = \calE(d_1-2)$
to see that 
\[
\chi(\calE(d_1-2)) = 
\begin{cases}
\frac 1{192} (c_1^2 - 4 c_2 - 9)(c_1^2 - 4 c_2 - 1), & \text{if}\; 2d_1= a_X+4;
\\
\frac 1{192}(c_1^2 - 4 c_2 - 4)(c_1^2 - 4 c_2 ), & \text{if}\; 2d_1= a_X+5;
\\
\frac 1{192}(c_1^2 - 4 c_2 - 1)(c_1^2 - 4 c_2 + 15 ) , & \text{if}\; 2d_1= a_X+6. 
\end{cases}
\]
Since 
$\chi(\calE(d_1-2))  \geq 0$, we obtain the proposition.
\end{proof}

Since $2d_1 \leq a_X+n+2$, $\calE$ need not be stable. 
Hence, \textit{a priori}, one does not know the sign of $c_1^2 - 4 c_2$; 
see~\cite[Chapter~2, \S 2.2]{OSSvb2011}.
The only possible upper bounds for $c_2$ that we can get from the above are
of the form $4c_2 \leq c_1^2$ minus a constant (that depends on $d_1$), and
not linear in $c_1$ such as the ones assumed by Holme-Schneider or Holme.

\section{The Horrocks-Mumford bundle}
\label{section:hmbundle}

In this section $\Bbbk = \complex$.
In \cite{HorrocksMumfordRank21973}, Horrocks and Mumford construct an
indecomposable rank-$2$ vector bundle $\calF$ on $\projective^4$, 
with $c_1(\calF) = 5$ and $c_2(\calF) = 10$. A general global section of
$\calF$ defines a smooth surface; call it $X$. Write $\calE = \calF^*$.
Then $X$ is subcanonical (using, e.g.,~\cite[Proposition
6.1]{HartshorneVarSmallCodim1974})
and we have an exact sequence
\[
0 \to \strSh_{\projective^4}(-5) \to \calE \to \calI_X \to 0. 
\]
We note from the cohomology table~\cite[p.~74]{HorrocksMumfordRank21973}
that $\homology^0(\calE(j)) = \homology^0(\calF(j-5)) = 0$ for all $j < 5$, and
$\rank_\complex \homology^0(\calE(5)) = 4$. 
This shows that $I_X$ does not have generators of degree less than $5$ and
that it has three generators in degree $5$.
Hence $d_1 = 5$. On the other hand $a_X + n+1  = c_1(\calF) = 5$. Thus 
$2d_1 > a_X + n+2$. Therefore neither Theorem~\ref{theorem:minHOne} nor
Corollary~\ref{corollary:gap} applies to $\calE$.
Moreover, using the cohomology table of $\calE$, we see that
\[
\homology^1(\projective^4,\, \calI_X(d_1 - 2)) = 
\homology^1(\projective^4,\, \calE(3)) = \complex^{10} \neq 0.
\]
Hence the remaining hypotheses of 
Theorem~\ref{theorem:minHOne} and
Corollary~\ref{corollary:gap} also do not hold for $X$.

\section*{Acknowledgements}

We thank the referees for a careful reading of the manuscript and for
suggesting that we compare our results with classical results.
The computer algebra system~\cite{M2} 
provided valuable assistance in studying examples.

\def\cfudot#1{\ifmmode\setbox7\hbox{$\accent"5E#1$}\else
  \setbox7\hbox{\accent"5E#1}\penalty 10000\relax\fi\raise 1\ht7
  \hbox{\raise.1ex\hbox to 1\wd7{\hss.\hss}}\penalty 10000 \hskip-1\wd7\penalty
  10000\box7}


\begin{thebibliography}{STCKS}

\bibitem[BC83]{BallicoChiantiniSmSubcan1983}
E.~Ballico and L.~Chiantini.
\newblock On smooth subcanonical varieties of codimension {$2$} in {${\bf
  P}^{n},$} {$n\geq 4$}.
\newblock {\em Ann. Mat. Pura Appl. (4)}, 135:99--117 (1984), 1983.

\bibitem[BH93]{BrHe:CM}
W.~Bruns and J.~Herzog.
\newblock {\em Cohen-{M}acaulay rings}, volume~39 of {\em Cambridge Studies in
  Advanced Mathematics}.
\newblock Cambridge University Press, Cambridge, 1993.

\bibitem[Chi06]{ChiantiniVBReflShLowCodimVars2006}
L.~Chiantini.
\newblock Vector bundles, reflexive sheaves and low codimensional varieties.
\newblock {\em Rend. Semin. Mat. Univ. Politec. Torino}, 64(4):381--405, 2006.

\bibitem[CV84]{ChiantiniVallabregaSubcanCurves1983}
L.~Chiantini and P.~Valabrega.
\newblock Subcanonical curves and complete intersections in projective
  {$3$}-space.
\newblock {\em Ann. Mat. Pura Appl. (4)}, 138:309--330, 1984.

\bibitem[CV87]{ChiantiniVallabregaSubcanCurves1987}
L.~Chiantini and P.~Valabrega.
\newblock On some properties of subcanonical curves and unstable bundles.
\newblock {\em Comm. Algebra}, 15(9):1877--1887, 1987.

\bibitem[EF02]{ElliaFrancoCodimTwo2002}
P.~Ellia and D.~Franco.
\newblock On codimension two subvarieties of {$\bold P^5$} and {$\bold P^6$}.
\newblock {\em J. Algebraic Geom.}, 11(3):513--533, 2002.

\bibitem[EG81]{EvansGriffithSyzProb1981}
E.~G. Evans and P.~Griffith.
\newblock The syzygy problem.
\newblock {\em Ann. of Math. (2)}, 114(2):323--333, 1981.

\bibitem[Eis95]{eiscommalg}
D.~Eisenbud.
\newblock {\em Commutative algebra, with a View Toward Algebraic Geometry},
  volume 150 of {\em Graduate Texts in Mathematics}.
\newblock Springer-Verlag, New York, 1995.

\bibitem[Ell00]{ElliaSurvey2000}
P.~Ellia.
\newblock A survey on the classification of codimension two subvarieties of
  {$\Bbb P^n$}.
\newblock {\em Matematiche (Catania)}, 55(2):285--303 (2002), 2000.
\newblock Dedicated to Silvio Greco on the occasion of his 60th birthday
  (Catania, 2001).

\bibitem[Eng07]{EnghetaPDUnmThreeCubics2007}
B.~Engheta.
\newblock On the projective dimension and the unmixed part of three cubics.
\newblock {\em J. Algebra}, 316(2):715--734, 2007.

\bibitem[ES99]{ElliaSartikBuchs1999}
P.~Ellia and A.~Sarti.
\newblock On codimension two {$k$}-{B}uchsbaum subvarieties of {${\bf P}^n$}.
\newblock In {\em Commutative algebra and algebraic geometry ({F}errara)},
  volume 206 of {\em Lecture Notes in Pure and Appl. Math.}, pages 81--92.
  Dekker, New York, 1999.

\bibitem[Ful98]{FultonIntThy1998}
W.~Fulton.
\newblock {\em Intersection theory}, volume~2 of {\em Ergebnisse der Mathematik
  und ihrer Grenzgebiete. 3. Folge. A Series of Modern Surveys in Mathematics
  [Results in Mathematics and Related Areas. 3rd Series. A Series of Modern
  Surveys in Mathematics]}.
\newblock Springer-Verlag, Berlin, second edition, 1998.

\bibitem[Gre89]{GreenTrieste89}
M.~L. Green.
\newblock Koszul cohomology and geometry.
\newblock In {\em Lectures on {R}iemann surfaces ({T}rieste, 1987)}, pages
  177--200. World Sci. Publ., Teaneck, NJ, 1989.

\bibitem[Har74]{HartshorneVarSmallCodim1974}
R.~Hartshorne.
\newblock Varieties of small codimension in projective space.
\newblock {\em Bull. Amer. Math. Soc.}, 80:1017--1032, 1974.

\bibitem[Har77]{HartAG}
R.~Hartshorne.
\newblock {\em Algebraic geometry}.
\newblock Springer-Verlag, New York, 1977.
\newblock Graduate Texts in Mathematics, No. 52.

\bibitem[Hir95]{HirzebruchTMAG1995}
F.~Hirzebruch.
\newblock {\em Topological methods in algebraic geometry}.
\newblock Classics in Mathematics. Springer-Verlag, Berlin, 1978 edition, 1995.
\newblock Translated from the German and Appendix One by R. L. E.
  Schwarzenberger, With a preface to the third English edition by the author
  and Schwarzenberger, Appendix Two by A. Borel.

\bibitem[HM73]{HorrocksMumfordRank21973}
G.~Horrocks and D.~Mumford.
\newblock A rank {$2$} vector bundle on {${\bf P}\sp{4}$} with {$15,000$}
  symmetries.
\newblock {\em Topology}, 12:63--81, 1973.

\bibitem[Hoc73]{HochsterContracted1973}
M.~Hochster.
\newblock Contracted ideals from integral extensions of regular rings.
\newblock {\em Nagoya Math. J.}, 51:25--43, 1973.

\bibitem[Hol89]{HolmeCodimTwo1989}
A.~Holme.
\newblock Codimension {$2$} subvarieties of projective space.
\newblock {\em Manuscripta Math.}, 65(4):427--446, 1989.

\bibitem[HS85]{HolmeSchnComputerAided1985}
A.~Holme and M.~Schneider.
\newblock A computer aided approach to codimension {$2$} subvarieties of {${\bf
  P}_n,\;n \geqq 6$}.
\newblock {\em J. Reine Angew. Math.}, 357:205--220, 1985.

\bibitem[KR00]{KumarRaoBuchsbaumBdles2000}
N.~M. Kumar and A.~P. Rao.
\newblock Buchsbaum bundles on {$\bold P^n$}.
\newblock volume 152, pages 195--199. 2000.
\newblock Commutative algebra, homological algebra and representation theory
  (Catania/Genoa/Rome, 1998).

\bibitem[Mat80]{MatsCA80}
H.~Matsumura.
\newblock {\em Commutative algebra}, volume~56 of {\em Mathematics Lecture Note
  Series}.
\newblock Benjamin/Cummings Publishing Co., Inc., Reading, Mass., second
  edition, 1980.

\bibitem[M2]{M2}
D.~R. Grayson and M.~E. Stillman.
\newblock Macaulay 2, a software system for research in algebraic geometry,
  2006.
\newblock Available at \url{http://www.math.uiuc.edu/Macaulay2/}.

\bibitem[Mig98]{MiglioreLiaisonBook1998}
J.~C. Migliore.
\newblock {\em Introduction to liaison theory and deficiency modules}, volume
  165 of {\em Progress in Mathematics}.
\newblock Birkh\"{a}user Boston, Inc., Boston, MA, 1998.

\bibitem[OSS11]{OSSvb2011}
C.~Okonek, M.~Schneider, and H.~Spindler.
\newblock {\em Vector bundles on complex projective spaces}.
\newblock Modern Birkh\"{a}user Classics. Birkh\"{a}user/Springer Basel AG,
  Basel, 1988 edition, 2011.
\newblock With an appendix by S. I. Gelfand.

\bibitem[Sch98]{SchenzelUseLC1998}
P.~Schenzel.
\newblock On the use of local cohomology in algebra and geometry.
\newblock In {\em Six lectures on commutative algebra ({B}ellaterra, 1996)},
  volume 166 of {\em Progr. Math.}, pages 241--292. Birkh\"{a}user, Basel,
  1998.

\bibitem[{STCKS}]{StacksProject}
{The Stacks project authors}.
\newblock The stacks project.
\newblock \url{https://stacks.math.columbia.edu}, 2018.

\end{thebibliography}
\end{document}